\documentclass[twoside,10pt,a4paper,reqno]{amsart}

\usepackage{amsmath, amsfonts,amsthm,amssymb,amsbsy,upref,color,graphicx,amscd,enumerate}
\usepackage[colorlinks=true,linkcolor=blue,citecolor=blue,%
filecolor=blue,urlcolor=blue,pageanchor=true,plainpages=false,%
linktocpage]{hyperref}
\usepackage[active]{srcltx}
\usepackage{mathrsfs}
\usepackage{tikz} 
\usetikzlibrary{trees}
\usepackage{caption}

\DeclareMathOperator{\diam}{diam}

\def\N{\mathbb{N}}
\def\R{\mathbb{R}}
\def\F{\mathcal{F}}
\def\H{\mathcal{H}}

\newcommand{\A}{\mathcal A}
\newcommand{\Z}{\mathbb Z}

\numberwithin{equation}{section}
\theoremstyle{plain}
\newtheorem{theorem}{Theorem}[section]
\newtheorem{lemma}[theorem]{Lemma}
\newtheorem{corollary}[theorem]{Corollary}
\newtheorem{proposition}[theorem]{Proposition}

\newtheorem{remark}[theorem]{Remark}

\title{Some remarks on nonexistence in the Gamow liquid drop model}

\author[D. Mazzoleni]{Dario Mazzoleni}
\address[Dario Mazzoleni]{Dipartimento di Matematica ``F.~Casorati'', Universit\`a di Pavia, Via Ferrata 5, 27100 Pavia, Italy}
\email[Dario Mazzoleni]{dario.mazzoleni@unipv.it}
\author[A. Pratelli]{Aldo Pratelli}
\address[Aldo Pratelli]{Dipartimento di Matematica\\
Universit\`a di Pisa\\
Largo Bruno Pontecorvo 5, 56127 Pisa, Italy }
\email[Aldo Pratelli]{aldo.pratelli@unipi.it}
\author[B. Ruffini]{Berardo Ruffini}
\address[Berardo Ruffini]{Dipartimento di Matematica\\
Universit\`a di Bologna\\
Piazza di Porta San Donato 5, 40126 Bologna, Italy }
\email[Berardo Ruffini]{berardo.ruffini@unibo.it}

\thanks{D. Mazzoleni and B. Ruffini have been partially supported by PRIN 2022 project 2022R537CS $NO^3$, CUP F53D23002810006, granted by the MUR and funded by the European Union -Next Generation EU, and by the INdAM-GNAMPA 2025 project - CUP E5324001950001.
The authors warmly wish to thank Andrea Moiola for some discussion about the numerical computation of the singular integrals for the proof of Theorem~\ref{thm:nonexcubes}. 
}

\begin{document}
	
\begin{abstract}
In this paper we deal with the celebrated Gamow
liquid drop model, namely, for $m>0$, $N\geq 2$ and $\alpha\in (0,N)$, we study the problem
\[
\min\Big\{P(\Omega)+ V_\alpha(\Omega) : \Omega\subset \R^N,\;|\Omega|=m\Big\}\,,
\]
and focus mostly on the nonexistence of minimizers for large masses. First we show that any possible minimizer in the asymptotic regime must satisfy an optimal thickness property. 
Eventually, in the simplified setting of sets made up as a tiling of cubes, we show that nonexistence occurs for large masses for the whole range of parameters $N\geq 2$ and $\alpha\in (0,N)$.
\end{abstract}

\maketitle

\section{Introduction}\label{sec:introduction}
The celebrated Gamow liquid drop model was introduced by Gamow in~\cite{Gam} to quantitatively describe the binding energy curve for atomic nuclei. After a rescaling of the physical constants, this consists in solving
\begin{equation}\label{eq:pbwithm}
\min\Big\{P(\Omega)+ V_\alpha(\Omega) : |\Omega|=m\Big\}\,.
\end{equation}
Here $\Omega$ denotes a measurable set in $\R^N$, $N\geq 2$, of measure $|\Omega|=m$, $P(\Omega)$ is its De Giorgi Perimeter and 
\begin{equation}\label{Rieszenergy}
V_\alpha(\Omega)=V_\alpha(\Omega,\Omega)=\int_\Omega\int_\Omega\frac{1}{|y-x|^{N-\alpha}}\,dxdy,\qquad \text{with} \qquad V_\alpha(E,F)=\int_E\int_F\frac{1}{|y-x|^{N-\alpha}}\,dx\,dy\,,
\end{equation}
is the Riesz (interaction) energy, for $\alpha\in (0,N)$.
By a rescaling argument, problem~\eqref{eq:pbwithm} is equivalent to 
\begin{equation}\label{eq:pbmin}
\inf\Big\{\mathcal F_{K,\alpha}(\Omega):=P(\Omega)+K V_\alpha(\Omega) : |\Omega|=|B_1|=\omega_N\Big\}\,,
\end{equation}
where $B_1$ is a ball of radius $1$ and $K=(\tfrac{m}{\omega_N})^\frac{1+\alpha}{N}$. Intuitively, since for small values of $m$ (or of $K$) the perimeter is the leading term, one expects that minimizers should exist and look more or less like balls. And in fact, this is well-known to be exact (see~\cite{KM1,KM2,FFMMM,jul14}); actually much more is true, namely, for $m$ small enough the unique minimizer is \emph{exactly} the ball. Conversely, for large values of the mass the leading term is the repulsive term $V_\alpha$, so one easily expects fragmentation of the mass and then nonexistence; one might even imagine much more, that is, that the minimizers exist only if they are a ball. Equivalently, one can say that ``generalized minimizers'' are unions of equal balls so that any two different balls have infinite distance (and so, no interaction). Concerning this second guess, the situation is much less clear. The fact that minimizers are equal balls at infinite distance has been established in~\cite{boncri} only in a very restricted range of exponents; instead, the nonexistence of minimizers for large mass (which is weaker than the ``balls at infinite property'') is now known for every dimension $N\in\N$ but only for exponents $\alpha\in [N-2,N)$, see~\cite{KM1,KM2,LuOtto,FrankKillipNam,FrankNam}. 
A full understanding of the problem in the physical case $N=3$, $\alpha=2$ has appeared only recently~\cite{ChodoshGianocca}: there is a physical, explicit threshold $K^*$ such that balls are the unique minimizer if $K\leq K^*$, while no minimizer exists for $K>K^*$.
We also point out the recent papers~\cite{DelPino,DelPinoFrankMusso} that deal with sets which are critical points of the Gamow energy for any value of $K$ (in the physical case $N=3$ and $\alpha=2$).

\subsection{Scope of the paper and main results}
The aim of this note is to shed some light on the behavior of the problem in the regime $K \gg 1$, with particular focus on the nonexistence issue; since as said above the nonexistence is known only for $\alpha \in [N-2,N)$, the unresolved range motivating the paper is $\alpha\in (0,N-2)$. 

In order to state our first result, we consider for a moment the ``full conjecture'', that is, that a generalized minimizer must be a union of equal balls with infinite mutual distances. Assuming that this conjecture is true, it is immediate to calculate the corresponding radius of these balls; namely, up to a multiplicative constant, the radius of the optimal balls for $\F_{K,\alpha}$ must be of order
\begin{equation}\label{eq:rk}
r_{K,\alpha}=K^{-\frac{1}{1+\alpha}}\,,
\end{equation}
see the discussion in Section~\ref{sec:prelim}. Note that the volume constraint then requires the number of balls to be of order $r_{K,\alpha}^{-N}$.
It is easy to check that, on this union of balls, there is \emph{equipartition of the energy}, namely,\[
P(\cup_i B_{r_{K,\alpha}}(x_i))\sim KV_\alpha(\cup_i B_{r_{K,\alpha}}(x_i))\sim K^{\frac{1}{1+\alpha}}.
\]
This may seem to be a peculiar property, but in fact it holds for all minimizers of $\F_{K,\alpha}$, as it was proved in~\cite{KM2} with interpolation inequalities, but may be shown also with elementary geometric arguments. Precisely, if $\Omega_K$ is an optimal set for~\eqref{eq:pbmin}, then
\begin{equation}\label{eq:equipartition}
P(\Omega_K)\sim KV_\alpha(\Omega_K)\sim K^{\frac{1}{1+\alpha}},
\end{equation}
up to constants depending only on $N,\alpha$.
This fact suggests that the quantity $r_{K,\alpha}$ may have a special role in the problem.

In fact, our first result says that if a minimizer exists, its ``thickness'' must indeed be of the same order as $r_{K,\alpha}$. More precisely, we have the following.

\begin{theorem}[Optimal thickness]\label{thm:thickness}
Let $N\geq 2$. There exists a constant $c=c(N,\alpha)$ such that for every $K\geq 1$, if $\Omega_K$ is an optimal set for~\eqref{eq:pbmin}, then the following facts hold.
\begin{itemize}
\item \emph{$\Omega_K$ is not too thick: }
$\Omega_K$ does not contain any ball of radius $cr_{K,\alpha}$.
\item \emph{$\Omega_K$ is not too thin: } 
For every ball $B_r(x)$ with radius $r<c^{-1}r_{K,\alpha}$, we have
\begin{equation}\label{eq:nosmallcuts}
\H^{N-1}\big(\Omega_K\cap \partial B_r(x)\big) \geq \frac 12 \, \H^{N-1} \big(\partial^*\Omega_K\cap B_r(x)\big)\,.
\end{equation}
\end{itemize}
\end{theorem}

The above claim deserves a quick comment: while the fact that $\Omega_K$ does not contain any ball of radius $c r_{K,\alpha}$ clearly says in a sense that ``$\Omega_K$ is not too thick'', it might be at first glance less clear why~(\ref{eq:nosmallcuts}) should mean that ``$\Omega_K$ is not too thin''. But in fact, that estimate is precisely saying that the ``size'' of $\Omega_K$ must be at least of order $r_{K,\alpha}$; indeed, (\ref{eq:nosmallcuts}) prevents $\Omega_K$ from having connected components compactly contained in a ball of radius $r_{K,\alpha}/c$, and more in general $\Omega_K$ cannot have tentacles with radius much smaller than $r_{K,\alpha}$ or similar structures. It is to be underlined that the constant $1/2$ in~(\ref{eq:nosmallcuts}) can be replaced by any other constant smaller than $1$. In light of the thickness result of Theorem~\ref{thm:thickness}, one can qualitatively guess the shape of possible minimizers: this is compatible, for instance, with cylinder-like or slab-like structures with thickness of order $r_{K,\alpha}$.\par

Our next aim is to investigate a simplified setting in which the mechanism leading to fragmentation can be analyzed explicitly. To this end, we restrict the admissible configurations to finite unions of congruent cubes on a lattice; although this model does not resolve the full problem, it allows us to isolate and quantify the competition between the perimeter gain produced by contacts and the repulsive interaction between the components.

In analogy with the conjectured picture of generalized minimizers as collections of components whose mutual distances diverge, we first consider isolated cubes and determine the side length which minimizes the energy per unit volume. Namely, setting
\begin{equation}\label{eq:gKalpha}
g_{K,\alpha}(r)
:=\frac{\F_{K,\alpha}(rQ)}{|rQ|}
=\frac{2N}{r}+K V_\alpha(Q)r^\alpha,
\end{equation}
we find the optimal side length
\begin{equation}\label{eq:rho}
\rho_{K,\alpha}
:=\left(\frac{2N}{\alpha K V_\alpha(Q)}\right)^{\frac{1}{1+\alpha}},
\end{equation}
where $Q=(0,1)^N$. Notice that
\[
\rho_{K,\alpha}\sim r_{K,\alpha}\sim K^{-\frac{1}{1+\alpha}},
\]
as suggested by the scaling of the original problem.

If one were to prescribe exactly the side length $\rho_{K,\alpha}$, however, the volume constraint $|E|=\omega_N$ would in general not be compatible with an integer number of congruent cubes. This does not affect the argument below, since the geometric nonattainment mechanism applies more generally whenever the common side length $\ell$ satisfies
\[
\ell\geq \rho_{K,\alpha}.
\]
We therefore change the point of view while keeping the total volume fixed. We fix an integer $M\geq2$ and consider unions of exactly $M$ congruent axis-parallel cubes. 
Precisely, we set
\begin{equation}\label{eq:cubical-class}
  \A_M:=
  \left\{\bigcup_{i=1}^M\ell_M(Q+a_i):
    a_1,\ldots,a_M\in\Z^N\ \text{distinct}\right\}.
\end{equation}
The condition $|E|=\omega_N$ then determines their side length uniquely as
\begin{equation}\label{eq:ell-fixed}
\ell_M:=\left(\frac{\omega_N}{M}\right)^{1/N}.
\end{equation}
For this fixed side length, the condition $\ell_M\geq\rho_{K,\alpha}$ can equivalently be written as a lower bound on $K$. Indeed,
\[
\ell_M\geq\rho_{K,\alpha}\qquad\text{if and only if}\qquad \ell_M^{1+\alpha}
\geq
\frac{2N}{\alpha K V_\alpha(Q)},
\]
which is equivalent to
\[
K\geq
\frac{2N}{\alpha V_\alpha(Q)}
\ell_M^{-(1+\alpha)},\qquad\text{that is,}\qquad 
K\geq
\frac{2N}{\alpha V_\alpha(Q)}
\left(\frac{M}{\omega_N}\right)^{\frac{1+\alpha}{N}}.
\]
This motivates the definition
\begin{equation}\label{eq:K0}
K_0(N,\alpha,M)
:=
\frac{2N}{\alpha V_\alpha(Q)}
\left(\frac{M}{\omega_N}\right)^{\frac{1+\alpha}{N}},
\end{equation}
so that the condition $\ell_M\geq\rho_{K,\alpha}$ is exactly equivalent to
$K\geq K_0(N,\alpha,M).$

We prove the following nonexistence result for the Gamow problem in the cubic setting in the large-$K$ regime. 

\begin{theorem}[Cubical nonexistence]\label{thm:nonexcubes}
  Let $N\geq2$, $\alpha\in (0,N)$, and let $M\geq2$ be an integer.
  Set $\ell=\ell_M$ as in \eqref{eq:ell-fixed}. Then for every
  \begin{equation}\label{eq:large-K-assumption}
    K\geq K_0(N,\alpha,M),
  \end{equation}
  one has
  \begin{equation}\label{eq:infimum}
    \inf_{E\in\A_M}\F_{K,\alpha}(E)
   =M\bigl(2N\ell^{N-1}
      +K\ell^{N+\alpha}V_\alpha(Q)\bigr),
  \end{equation}
  and the infimum is not attained. It is approached by configurations in
  which all mutual distances between the $M$ cubes tend to infinity.
\end{theorem}

The main analytic ingredient is an exact comparison between the
self-interaction of a cube and the interaction of two cubes sharing a face.
Writing $e_1=(1,0,\ldots,0)$, we show in
Proposition~\ref{prop:identity} that
\begin{equation}\label{eq:identity-intro}
  V_\alpha(Q,Q+e_1)=\frac{\alpha}{2N}V_\alpha(Q)+R_{N,\alpha},
  \qquad R_{N,\alpha}>0.
\end{equation}
The proof uses the product structure of $Q$ and integration by parts with
the vector field $z|z|^{\alpha-N}$. 

At the fixed side length $\ell=\ell_M$, the definition of $K_0$ gives
\begin{equation}\label{eq:threshold-intro}
  K_0\ell^{1+\alpha}V_\alpha(Q)=\frac{2N}{\alpha}.
\end{equation}
Combining \eqref{eq:threshold-intro} with \eqref{eq:identity-intro} yields
\begin{equation}\label{eq:strict-contact-intro}
  K\ell^{1+\alpha}V_\alpha(Q,Q+e_1)
  \geq \frac{2N V_\alpha(Q,Q+e_1)}{\alpha V_\alpha(Q)}>1
  \qquad\text{for every }K\geq K_0.
\end{equation}
A shared face reduces the perimeter by $2\ell^{N-1}$, whereas the mutual
interaction of the two touching cubes contributes
$2K\ell^{N+\alpha}V_\alpha(Q,Q+e_1)$. Hence every face contact is strictly
unfavourable for $K\geq K_0$. All other pairs give a strictly positive
interaction and no perimeter reduction. This pairwise comparison is the
mechanism behind Theorem~\ref{thm:nonexcubes}.

The plan of the paper is very simple. First, in Section~\ref{sec:prelim} we fix the notation and collect some simple estimates. Then, Section~\ref{sec:main} is devoted to prove Theorem~\ref{thm:thickness}. Finally, in Section~\ref{sec:cubes} we show Theorem~\ref{thm:nonexcubes}.

\section{Notations and preliminaries\label{sec:prelim}} 

This section is devoted to fix the notation and to present some preliminary observations and estimates. Through all the paper we will denote by $N\geq 2$ the dimension of the ambient space, and by $\alpha$ a given constant in $(0,N)$ that we use to define the Riesz energy as in~(\ref{Rieszenergy}); for any $K\geq 1$, we consider the energy $\F_{K,\alpha}$ as in~(\ref{eq:pbmin}). The elements of the canonical basis of $\R^N$ will be denoted as $e_i$ for $1\leq i\leq N$. Given any set $E\subseteq\R^N$, we will denote by $|E|=\H^N(E)$ its volume and by $P(E)=\H^{N-1}(\partial^* E)$ its perimeter (in the De Giorgi sense, see for instance~\cite{AFP}); for any open set $A\subseteq\R^N$, the relative perimeter of $E$ in $A$ will be denoted as $P(E;A) = \H^{N-1}(\partial^* E\cap A)$.

For any $r>0$ and any $x\in\R^N$, the ball with radius $r$ and center $x$ will be denoted as $B_r(x)$, or just $B_r$ if the center is the origin, or $B$ if the ball has unit volume. By rescaling, for every $\lambda>0$ and every $E,F\subset\R^N$ we have
\begin{align}\label{rescaling}
P(\lambda E) = \lambda^{N-1} P(E)\,, && V_\alpha(\lambda E,\lambda F) = \lambda^{N+\alpha} V_\alpha(E,F)\,.
\end{align}

As described in the Introduction, it is conjectured that for $K$ large the problem~(\ref{eq:pbmin}) has no solution, and that a minimizing sequence is given by equal balls of some optimal radius having exploding mutual distances; equivalently, we can say that a ``generalized solution'' is given by equal balls at infinite distance (for a formal description of these generalized solutions see for instance~\cite{NovagaPratelli}). Let us then restrict our attention, for a moment, to a set $\Omega$ made by $H$ equal balls, having very large mutual distances. Then, neglecting the mutual interaction between different balls (which is arbitrarily small up to take the balls distant enough) we have
\[\begin{split}
\F_{K,\alpha}(\Omega) &= P(\Omega) + KV_\alpha(\Omega)
= H \Big(P(B_{H^{-1/N}}) + K V_\alpha(B_{H^{-1/N}})\Big)
= H \Big( N\omega_N H^{-\frac{N-1}N} + K H^{-\frac{N+\alpha}N}V_\alpha(B)\Big)\\
&= N\omega_N H^{1/N} + K H^{-\frac \alpha N} V_\alpha(B)\,.
\end{split}\]
As a consequence, among the union of equal balls, the best one has a number of balls which is $H \sim K^{\frac N{1+\alpha}}$; the corresponding ``best radius'' is then of order $K^{-\frac 1{1+\alpha}}$, so it corresponds (up to multiplicative constants) to the radius $r_{K,\alpha}$ defined in~(\ref{eq:rk}). 

Now, we can show that minimizers of $\F_{K,\alpha}$ for $K$ large, if they exist, must have very large diameter. The result is the following.
\begin{lemma}\label{lemmadiam}
Let $K\geq 1$ and let $\Omega_K$ be such that $\F_{K,\alpha}(\Omega_K)< \zeta\inf$\eqref{eq:pbmin}, being $\zeta>1$ given. Then, for any unit cube $Q$ we have
\begin{equation}\label{firstone}
\big| \Omega_K \cap Q\big| \lesssim K^{-\frac \alpha{2(1+\alpha)}} \,.
\end{equation}
As a consequence,
\begin{equation}\label{secondone}
\diam(\Omega_K)\gtrsim K^{\frac{\alpha}{(1+\alpha)(N-\alpha)}}\,.
\end{equation}
\end{lemma}
\begin{proof}
We have already noticed that a union of balls of radius of order $K^{\frac1 {1+\alpha}}$ and very far from each other have energy of order $K^{\frac 1{1+\alpha}}$, and being $\Omega_K$ optimal (up to a factor $\zeta$) it must clearly be $KV_\alpha(\Omega_K) \leq \F_{K,\alpha}(\Omega_K)\lesssim K^{\frac 1{1+\alpha}}$, which means $V_\alpha(\Omega_K) \lesssim K^{-\frac \alpha{1+\alpha}}$. We have then
\[
K^{-\frac \alpha{1+\alpha}} \gtrsim V_\alpha(\Omega_K) \geq \int_{\Omega_K\cap Q}\int_{\Omega_K\cap Q} \frac 1{|y-x|^{N-\alpha}}\,dx\,dy\gtrsim  \big|\Omega_K\cap Q\big|^2\,,
\]
and this is~(\ref{firstone}). The same argument, substituting the unit cube $Q$ with a cube $\widetilde Q$ of diameter $d$, gives
\[
K^{-\frac \alpha{1+\alpha}} \gtrsim \big|\Omega_K\cap \widetilde Q\big|^2 \, d^{\alpha-N}\,,
\]
and this clearly implies~(\ref{secondone}).
\end{proof}

\section{Proof of Theorem~\ref{thm:thickness}}\label{sec:main}

This section is devoted to the proof of Theorems~\ref{thm:thickness}.

\begin{proof}[Proof (of Theorem~\ref{thm:thickness})]
We start with the first part, namely, that if $c$ is large enough then for every $K\geq 1$ an optimal set $\Omega_K$ cannot contain a ball with radius $c r_{K,\alpha}$. Notice that, up to increase the value of $c$, we are allowed to assume that $c r_{K,\alpha}$ is quite small. More precisely suppose that we have found $c_0$ such that for every $K\geq 1$ and every optimal set $\Omega_K$ we have that $\Omega_K$ cannot contain a ball of radius $c_0 r_{K,\alpha}$ provided that $c_0 r_{K,\alpha}\leq \ell$ for some given small $\ell\leq 1$. But then, the theorem is true with $c=c_0/\ell$; in fact, if $c_0 r_{K,\alpha}\leq \ell$ then $\Omega_K$ cannot contain a ball of radius $c_0 r_{K,\alpha}$, and then it cannot contain a ball or radius $c r_{K,\alpha}$. Instead, if $c_0 r_{K,\alpha}>\ell$, then $c r_{K,\alpha}>1$ and then for sure $\Omega_K$ cannot contain a ball of radius $c r_{K,\alpha}$, simply because such a ball has voume strictly larger than $\omega_N=|\Omega_K|$.\par

Now, suppose by contradiction that for some $K\geq 1$ an optimal set $\Omega_K$ contains the ball $B_{c r_{K,\alpha}}$; then, we define the competitor
\[
\widetilde\Omega = \big(1- (c r_{K,\alpha})^N\big)^{-\frac 1N} \Big(\Omega_K\setminus B_{c r_{K,\alpha}}\Big)\,,
\]
which satisfies $|\widetilde\Omega|=\omega_N$. Since $c r_{K,\alpha}$ is small, we have
\[\begin{split}
P(\widetilde\Omega) &= \big(1- (c r_{K,\alpha})^N\big)^{-\frac {N-1}N} \big(P(\Omega_K)+N\omega_N (c r_{K,\alpha})^{N-1}\big)
\leq \big(1 + (c r_{K,\alpha})^N \big)\,\big(P(\Omega_K)+N\omega_N (c r_{K,\alpha})^{N-1}\big)\\
&\leq P(\Omega_K) + (cr_{K,\alpha})^N P(\Omega_K) + 2 N\omega_N (cr_{K,\alpha})^{N-1}\,,
\end{split}\]
while
\[\begin{split}
V_\alpha(\widetilde\Omega) &= \big(1- (c r_{K,\alpha})^N\big)^{-\frac {N+\alpha}N} V_\alpha\big(\Omega_K\setminus B_{c r_{K,\alpha}}\big)
\leq \big( 1 + 2 (c r_{K,\alpha})^N\big)\, \big( V_\alpha(\Omega_K)- V_\alpha(B_{c r_{K,\alpha}})\big)\\
&\leq V_\alpha(\Omega_K) +  2 (c r_{K,\alpha})^N V_\alpha(\Omega_K)- (c r_{K,\alpha})^{N+\alpha}V_\alpha(B)\,.
\end{split}\]
We get a contradiction if $\F_{K,\alpha}(\widetilde\Omega)<\F_{K,\alpha}(\Omega_K)$, and putting together the last two estimates this is true if
\[
(cr_{K,\alpha})^N P(\Omega_K) + 2 N\omega_N (cr_{K,\alpha})^{N-1} +  2 (c r_{K,\alpha})^N K V_\alpha(\Omega_K)\leq K (c r_{K,\alpha})^{N+\alpha}V_\alpha(B)\,.
\]
Since we can choose a large $c$ and the highest power of $c$ in the above inequality is on the right, and keeping in mind that $V_\alpha(B)$ is a constant, then it is enough to show that
\begin{align*}
r_{K,\alpha}^N \F_{K,\alpha}(\Omega_K) \lesssim K r_{K,\alpha}^{N+\alpha} \,, && r_{K,\alpha}^{N-1} \lesssim K r_{K,\alpha}^{N+\alpha}\,.
\end{align*}
The definition~\eqref{eq:rk} of $r_{K,\alpha}$ gives $K r_{K,\alpha}^{N+\alpha}=r_{K,\alpha}^{N-1}$, so the second estimate is emptily true, and also the first estimate is true since $\F_{K,\alpha}(\Omega_K)\lesssim r_{K,\alpha}^{-1}=K^{\frac 1{1+\alpha}}$, as already noticed.

Let us now pass to consider the second part. Let $\Omega_K$ be an optimal set for~(\ref{eq:pbmin}), and let us consider a ball with radius $r<c^{-1} r_{K,\alpha}$, that for simplicity of notations we assume to be the ball $B_r$ centered at the origin. We can define the competitor
\[
\widetilde \Omega = \lambda \big(\Omega_K\setminus B_r\big)\,,
\]
where $\lambda$ is the constant (slightly larger than $1$) such that $|\widetilde\Omega|=|\Omega_K|=\omega_N$, thus
\[
\lambda = \bigg(\frac{\omega_N}{\omega_N - v}\bigg)^{1/N}\,,
\]
where we denote for brevity $v=|\Omega_K\cap B_r|$. If~(\ref{eq:nosmallcuts}) is false, then
\[
\F_{K,\alpha}(\Omega_K\setminus B_r) < \F_{K,\alpha}(\Omega_K) +  \big(P(\Omega_K\setminus B_r) - P(\Omega_K)\big)
\leq \F_{K,\alpha}(\Omega_K) - \frac 13\, P(\Omega_K\cap B_r)\leq \F_{K,\alpha}(\Omega_K) - \frac{N\omega_N^{1/N}}3\, v^{\frac{N-1}N}\,.
\]
Keeping in mind the rescaling formulas~(\ref{rescaling}) and the fact that $\lambda\geq 1$, we get
\[\begin{split}
\F_{K,\alpha}(\widetilde\Omega) &\leq \lambda^{N+\alpha} \F_{K,\alpha}(\Omega_K\setminus B_r) \leq \bigg( 1 + \frac{4v}{\omega_N}\bigg)\, \bigg(\F_{K,\alpha}(\Omega_K) - \frac {N\omega_N^{1/N}}3\, v^{\frac{N-1}N}\bigg)\\
&\leq \F_{K,\alpha}(\Omega_K) +\frac{4v}{\omega_N} \F_{K,\alpha}(\Omega_K) -  \frac {N\omega_N^{1/N}}3\, v^{\frac{N-1}N}\,.
\end{split}\]
Recalling once again that $\F_{K,\alpha}(\Omega_K)\lesssim r_{K,\alpha}^{-1}$, this estimate provides us the required contradiction as soon as $r_{K,\alpha} \gtrsim v^{\frac 1N}$, and since by construction $v\leq \omega_N(c^{-1} r_{K,\alpha})^N$ we obtain the thesis as soon as $c$ has been chosen large enough.
\end{proof}

Notice that in the above proof the constant $c_{N,\alpha}$ can be explicitly traced. We now notice a simple consequence of Lemma~\ref{lemmadiam}, which will not be needed in the rest of the paper but it is an interesting and elementary ``cutting property''.

\begin{corollary}\label{corcut}
Let $\Omega_K$ be a set such that $\F_{K,\alpha}(\Omega_K) < \zeta \inf\eqref{eq:pbmin}$, for some given $\zeta>1$. Then, provided that $K$ is large enough, there exists a hyperplane $H$ such that, calling $H^\pm$ the two parts in which $\R^N$ is divided by $H$, one has
\begin{align}\label{goodest}
\H^{N-1}\big(\Omega_K \cap H\big)\leq 1\,, && \big| \Omega_K \cap H^\pm\big| \in \bigg( \frac {\omega_N}{3N},\, \omega_N-\frac {\omega_N}{3N}\bigg)\,.
\end{align}
\end{corollary}
\begin{proof}
Let us define $a_1<b_1$ as
\begin{align*}
a_1 = \sup \bigg\{ t\in \R,\, \big| \Omega_K \cap \{ x_1 \leq t\} \big| \leq \frac{\omega_N}{3N}\bigg\}\,, &&
b_1 = \inf \bigg\{ t\in \R,\, \big| \Omega_K \cap \{ x_1 \geq t\} \big| \leq \frac{\omega_N}{3N}\bigg\}\,,
\end{align*}
so that the part of $\Omega_K$ which is not contained in the slab $\Pi_1 = [a_1,\, b_1] \times \R^{N-1}$ has volume $2\omega_N/3N$. If the width $b_1-a_1$ of the slab is greater than $\omega_N$, then there is some $a_1<t<b_1$ such that
\[
\H^{N-1}\big( \Omega_K \cap \{ x_1=t\}\big) \leq 1\,,
\]
and then the proof is obtained with $H= \{x_1=t\}$. If we repeat the same argument with slabs $\Pi_j$ with $2\leq j\leq N$, we conclude the proof unless each of the slabs has width smaller than $\omega_N$. In other words, if the claim is false then the intersection of the slabs is contained in a cube $Q$ of side $\omega_N$; moreover, by construction the intersection of this cube with $\Omega_K$ has volume at least $\omega_N/3$. And this is clearly in contradiction with Lemma~\ref{lemmadiam} as soon as $K$ is large enough.
\end{proof}

\section{Proof of Theorem~\ref{thm:nonexcubes}}\label{sec:cubes}

Thanks to the scaling highlighted in Section~\ref{sec:prelim}, for $r>0$, we have
\begin{equation}\label{eq:cube-energy}
  \F_{K,\alpha}(rQ)
  =2Nr^{N-1}+K r^{N+\alpha}V_\alpha(Q).
\end{equation}
For the prescribed-volume class \eqref{eq:cubical-class}, the side length
is $r=\ell (=\ell_M)$ and is therefore completely determined by $\omega_N$ and the
integer $M$.

Note also that  \eqref{eq:K0} is equivalent to
\begin{equation}\label{eq:K0-equivalent}
  K_0\ell^{1+\alpha}V_\alpha(Q)=\frac{2N}{\alpha}.
\end{equation}
Recalling the energy per unit volume of a cube of side length $r$ from~\eqref{eq:gKalpha}, $K_0$ is precisely the value for which the fixed side length
$\ell$ satisfies the stationary scaling relation for the isolated-cube
specific energy. 

\begin{remark}\label{rmk:infimum}
Let us also show that the asymptotic value of the functional among disjoint union of cubes of sidelength $\ell$ at reciprocal diverging distance is exactly $M\F_{K,\alpha}(\ell Q)$, that is, the right hand side of~\eqref{eq:infimum}.
For integers $n\geq2$, we define
  \begin{equation}\label{eq:separating-sequence}
    E_n:=\bigcup_{i=0}^{M-1}\ell(Q+in e_1).
  \end{equation}
  Every $E_n$ belongs to $\A_M$, and the cubes have no shared faces.
  For $0\leq i<j\leq M-1$, their first coordinates are separated by at
  least $\ell((j-i)n-1)$. Since $N-\alpha>0$,
  \[
    V_\alpha\bigl(\ell(Q+in e_1),\ell(Q+jn e_1)\bigr)
    \leq\frac{\ell^{N+\alpha}}{(n-1)^{N-\alpha}}.
  \]
  It follows that, as $n\to+\infty$,
  \[
    0\leq\F_{K,\alpha}(E_n)-M\F_{K,\alpha}(\ell Q)
    \leq \frac{K M(M-1)\ell^{N+\alpha}}{(n-1)^{N-\alpha}}
    \longrightarrow0.
  \]
\end{remark}

\subsection{A technical identity}\label{subsec:technical}
We develop here an exact identity for two cubes sharing a face, by a sort of calibration method.
Define
\begin{equation}\label{eq:tent}
  h(t):=(1-|t|)_+,
  \qquad H(z):=\prod_{i=1}^N h(z_i).
\end{equation}
Since $H(z)=|Q\cap(Q-z)|$, Fubini's theorem gives
\begin{equation}\label{eq:self-convolution}
  V_\alpha(Q)=\int_{(-1,1)^N}H(z)|z|^{\alpha-N}\,dz.
\end{equation}

The proof of the next result is based on the autocorrelation representation of the Riesz energy and on an integration by parts argument exploiting the homogeneity of the Riesz kernel. Related representations of Riesz energies in terms of overlap functions can be found in~\cite{Gotz}, while integration by parts arguments for Riesz-type energies in a geometric setting have been used, for example, recently in~\cite{BonaciniCristoferiTopaloglu}.

\begin{proposition}\label{prop:identity}
  Let $N\geq2$ and $\alpha\in (0,N)$. Then
  \begin{equation}\label{eq:identity}
    V_\alpha(Q,Q+e_1)=\frac{\alpha}{2N}V_\alpha(Q)+R_{N,\alpha},
  \end{equation}
  where, calling $z'=(z_2,\ldots,z_N)$,
  \begin{equation}\label{eq:remainder}
    R_{N,\alpha}:=
    \int_{(-1,1)^{N-1}}\prod_{j=2}^N h(z_j)
      \int_1^2\frac{2-s}{(s^2+|z'|^2)^{(N-\alpha)/2}}\,ds\,dz'
    >0.
  \end{equation}
\end{proposition}

\begin{proof}
Let $X(z):=z|z|^{\alpha-N}$ for $z\ne0$ and $H$ be as in~\eqref{eq:tent}. A direct calculation gives
  \[
    \operatorname{div}X(z)=\alpha|z|^{\alpha-N}
    \qquad\text{for $z\not=0$}.
  \]

%
%
%
\noindent For $0<\varepsilon<1$, set
$D_\varepsilon:=(-2,2)^N\setminus\overline{B_\varepsilon}$.
Since $HX$ is Lipschitz on $\overline{D_\varepsilon}$,
the divergence theorem, together with $H=0$ on $\partial(-2,2)^N$
and $\nu_{D_\varepsilon}(z)=-z/\varepsilon$ on $\partial B_\varepsilon$,
gives
\begin{align*}
	\alpha\int_{D_\varepsilon}H(z)|z|^{\alpha-N}\,dz
	+\int_{D_\varepsilon}X(z)\cdot\nabla H(z)\,dz
	&=\int_{\partial D_\varepsilon}HX\cdot\nu_{D_\varepsilon}
	\,d\mathcal H^{N-1}=-\varepsilon^{\alpha-N+1}
	\int_{\partial B_\varepsilon}H\,d\mathcal H^{N-1}.
\end{align*}
Using $0\leq H\leq1$ and $|\nabla H|\leq\sqrt N$ almost everywhere,
we estimate   the terms
\begin{align*}
	\varepsilon^{\alpha-N+1}
	\int_{\partial B_\varepsilon}H\,d\mathcal H^{N-1}
	&\leq N\omega_N\varepsilon^\alpha,\\
	\alpha\int_{B_\varepsilon}H(z)|z|^{\alpha-N}\,dz
	&\leq\alpha N\omega_N\int_0^\varepsilon r^{\alpha-1}\,dr
	=N\omega_N\varepsilon^\alpha,\qquad\text{and}\\
	\int_{B_\varepsilon}|X(z)\cdot\nabla H(z)|\,dz
	&\leq N^{3/2}\omega_N\int_0^\varepsilon r^\alpha\,dr
	=\frac{N^{3/2}\omega_N}{\alpha+1}\varepsilon^{\alpha+1}.
\end{align*}
All three terms tend to zero as $\varepsilon\downarrow0$, since
$\alpha>0$. Recalling \eqref{eq:self-convolution} and the support of
$H$, we obtain
\begin{equation}\label{eq:ibp}
	\alpha V_\alpha(Q)
	=-\int_{\R^N}X(z)\cdot\nabla H(z)\,dz.
\end{equation}
Now note that, for almost every $z\in\R^N$,
  \[
    \partial_iH(z)
    =-\operatorname{sgn}(z_i)\mathbf1_{\{|z_i|<1\}}
       \prod_{j\ne i}h(z_j).
  \]
  Substituting in \eqref{eq:ibp} gives
  \begin{equation}\label{eq:coordinate-sum}
    \alpha V_\alpha(Q)
    =\sum_{i=1}^N\int_{(-1,1)^N}
       |z_i|\prod_{j\ne i}h(z_j)|z|^{\alpha-N}\,dz.
  \end{equation}
  The $N$ summands are equal by coordinate symmetry. Using also evenness
  in $z_1$, we conclude that
  \begin{equation}\label{eq:I}
  \begin{split}
    \frac{\alpha}{2N}V_\alpha(Q)&=I_{\alpha,N},\\
    I_{\alpha,N}&:=\int_{(-1,1)^{N-1}}\prod_{j=2}^N h(z_j)
       \int_0^1\frac{s}{(s^2+|z'|^2)^{(N-\alpha)/2}}\,ds\,dz'.
  \end{split}
  \end{equation}
  These integrals are finite: near zero, the integrands in
  \eqref{eq:coordinate-sum} are bounded by $|z|^{\alpha-N+1}$.
  We now compute $V_\alpha(Q,Q+e_1)$. For $x\in Q$, $y\in Q+e_1$, set
  $z=y-x$. 
  Making the change of variables $s=y_1-x_1$, we observe that, for each fixed $s$, the admissible values of $x_1$ belong to the interval
\[
(0,1)\cap(1-s,2-s),
\]
whose length is
  \[
    a(s):=
    \begin{cases}
      s,&0<s<1,\\
      2-s,&1\leq s<2,\\
      0,&\text{otherwise}.
    \end{cases}
  \]
  The corresponding density in every other coordinate is $h$. Therefore
  \begin{align*}
    V_\alpha(Q,Q+e_1)
    &=\int_{(-1,1)^{N-1}}\prod_{j=2}^N h(z_j)
       \int_0^2\frac{a(s)}{(s^2+|z'|^2)^{(N-\alpha)/2}}\,ds\,dz'\\
    &=I_{\alpha,N}+R_{N,\alpha}.
  \end{align*}
  Together with \eqref{eq:I}, this proves \eqref{eq:identity}.
  The remainder is finite because $s\geq1$, and it is strictly positive
  because its integrand is positive for $1<s<2$ and
  $z'\in(-1,1)^{N-1}$.
\end{proof}

\subsection{Perimeter decomposition and proof of nonexistence}

Fix an integer $M\geq2$ and abbreviate $\ell:=\ell_M$. Let
$E\in\A_M$ and write
$E=\bigcup_{i=1}^M C_i$, with $C_i=\ell(Q+a_i)$. Let
\[
  S(E):=\{(i,j):1\leq i<j\leq M,\ |a_i-a_j|=1\},
  \qquad c(E):=\#S(E).
\]
Thus $c(E)$ counts the shared faces, each counted once. Throughout this
section, an empty sum is understood to be zero.

\begin{lemma}\label{lem:decomposition}
  For every $E\in\A_M$,
  \begin{equation}\label{eq:perimeter}
    P(E)=2NM\ell^{N-1}-2c(E)\ell^{N-1}
  \end{equation}
  and
  \begin{equation}\label{eq:riesz-decomposition}
    V_\alpha(E)=M\ell^{N+\alpha}V_\alpha(Q)
                  +2\sum_{1\leq i<j\leq M}V_\alpha(C_i,C_j).
  \end{equation}
  Consequently,
  \begin{equation}\label{eq:energy-decomposition}
  \begin{split}
    \F_{K,\alpha}(E)-M\F_{K,\alpha}(\ell Q)
    ={}&2c(E)\ell^{N-1}
           \bigl(K\ell^{1+\alpha}V_\alpha(Q,Q+e_1)-1\bigr)+2K\!\sum_{\substack{1\leq i<j\leq M\\(i,j)\notin S(E)}}
           V_\alpha(C_i,C_j).
  \end{split}
  \end{equation}
\end{lemma}

\begin{proof}
  The sum of the individual cube perimeters is $2NM\ell^{N-1}$.
  At each shared face, the two copies of that face are absent from the
  reduced boundary of the union. Their total area is $2\ell^{N-1}$.
  All other intersections of cube closures have dimension at most $N-2$
  and hence zero $\mathcal H^{N-1}$ measure. This proves
  \eqref{eq:perimeter}. 

  Splitting the double integral over the cubes gives
  \eqref{eq:riesz-decomposition}. The diagonal terms are equal by
  translation invariance and scaling, and each unordered pair contributes
  twice its mutual interaction. If $(i,j)\in S(E)$, invariance under
  coordinate permutations and reflections yields
  \[
    V_\alpha(C_i,C_j)=\ell^{N+\alpha}V_\alpha(Q,Q+e_1).
  \]
  Separating these terms proves \eqref{eq:energy-decomposition}.
\end{proof}

\begin{proof}[Proof of Theorem~\ref{thm:nonexcubes}]
  Let $K\geq K_0(N,\alpha,M)$. By \eqref{eq:K0-equivalent} and
  Proposition~\ref{prop:identity},
  \begin{equation}\label{eq:strict-threshold}
  \begin{split}
    K\ell^{1+\alpha}V_\alpha(Q,Q+e_1)
    &\geq K_0\ell^{1+\alpha}V_\alpha(Q,Q+e_1)\\
    &=\frac{2N V_\alpha(Q,Q+e_1)}{\alpha V_\alpha(Q)}
      =1+\frac{2N R_{N,\alpha}}{\alpha V_\alpha(Q)}>1.
  \end{split}
  \end{equation}
Hence both terms on the right-hand side of
  \eqref{eq:energy-decomposition} are nonnegative. If $c(E)>0$, the first
  term is strictly positive. If $c(E)=0$, the second term contains all
  $M(M-1)/2$ pairs, each with strictly positive interaction. Since $M\geq2$,
  in either case
  \begin{equation}\label{eq:strict-lower-bound}
    \F_{K,\alpha}(E)>M\F_{K,\alpha}(\ell Q).
  \end{equation}
On the other hand, thanks to Remark~\ref{rmk:infimum}, we can infer that
  \[
    \inf_{E\in\A_M}\F_{K,\alpha}(E)
    =M\F_{K,\alpha}(\ell Q),
  \]
  which is exactly \eqref{eq:infimum}, recalling \eqref{eq:cube-energy}. The
  strict inequality \eqref{eq:strict-lower-bound} proves that the infimum
  is not attained.
 \end{proof}

\begin{remark}\label{rmk:Kl}
  Since the prescribed mass is $\omega_N$ and the integer $M$ is fixed, the
  side length $\ell=(\omega_N/M)^{1/N}$ is fixed. Hence the dimensionless
  parameter governing a face contact is
  $K\ell^{1+\alpha}V_\alpha(Q,Q+e_1)$, which grows linearly with $K$.
  Theorem~\ref{thm:nonexcubes} therefore describes a true strong-repulsion regime.
  No mass is sent to infinity as $K$ decreases, and no change in the number
  of cubes is hidden in the statement.
\end{remark}

\begin{remark}
The contribution of a single face contact in
\eqref{eq:energy-decomposition} is positive precisely when
\[
K\ell^{1+\alpha}V_\alpha(Q,Q+e_1)>1.
\]
Recalling the definition of $\ell$~\eqref{eq:ell-fixed}, the corresponding contact threshold is
\[
K_{\mathrm c}(N,\alpha,M)
:=
\frac{1}{V_\alpha(Q,Q+e_1)}
\left(\frac{M}{\omega_N}\right)^{\frac{1+\alpha}{N}}.
\]
Proposition~\ref{prop:identity} yields
\[
V_\alpha(Q,Q+e_1)
>
\frac{\alpha}{2N}V_\alpha(Q),
\]
and therefore
\[
K_{\mathrm c}(N,\alpha,M)<K_0(N,\alpha,M).
\]
In fact, the proof of Theorem~\ref{thm:nonexcubes} gives nonattainment
for every $K>K_{\mathrm c}$. We state the theorem with the larger threshold
$K_0$ because it arises naturally from the equivalent condition
$\ell\geq\rho_{K,\alpha}$ and depends only on the self-interaction
$V_\alpha(Q)$. Moreover, the strict remainder $R_{N,\alpha}>0$ in
Proposition~\ref{prop:identity} ensures that the required strict inequality
still holds at $K=K_0$. No optimality of $K_0$ as a nonexistence threshold
is claimed.
\end{remark}

%

\section*{Note on the use of AI }
We started working on this paper in 2023 and we presented the ideas contained in it at some workshops. However, we were able to show the strict positivity of the remainder $R_{N,\alpha}$ in the equality~\eqref{eq:identity} only for $\alpha\in (N-2-\delta, N)$ for a certain $\delta>0$, thus having a weaker result than Theorem~\ref{thm:nonexcubes}, which is valid for all $\alpha\in (0, N)$. The technical trick in Subsection  \ref{subsec:technical} yielding to the stronger present version of the result  has been obtained through interactions with AI.

\bibliographystyle{acm}
\bibliography{MPR}

\end{document}